\documentclass[11pt,a4paper,reqno]{amsart}

\usepackage{amssymb,amsmath,mathrsfs}
\usepackage{enumitem}
\usepackage{amscd}
\usepackage{verbatim}
\usepackage{xcolor}

\numberwithin{equation}{section}

\theoremstyle{plain}

\newtheorem{theorem}{Theorem}[section]
\newtheorem{proposition}[theorem]{Proposition}
\newtheorem{lemma}[theorem]{Lemma}
\newtheorem{corollary}[theorem]{Corollary}

\theoremstyle{definition}

\newtheorem{definition}[theorem]{Definition}
\newtheorem{remark}[theorem]{Remark}

\newtheorem{example}[theorem]{Example}
\newtheorem{examples}[theorem]{Examples}

\begin{document}

\title{Parameter-dependent stochastic optimal control in finite discrete time}

\author{Asgar Jamneshan}
\address{Department of Mathematics and Statistics, University of Konstanz}
\email{asgar.jamneshan@uni-konstanz.de}

\author{Michael Kupper}
\address{Department of Mathematics and Statistics, University of Konstanz}
\email{kupper@uni-konstanz.de}

\author{Jos\'e Miguel Zapata-Garc\'ia}
\address{Department of Mathematics, Universidad de Murcia}
\email{jmzg1@um.es}

\thanks{The authours would like to thank Ilya Molchanov for discussions on Section \ref{sec:randomset}. The first two authors gratefully acknowledge financial support from DFG-Project KU 2740/2-1. The third author  was supported by a grant associated to the project MTM2014-57838-C2-1-P (MINECO)}


\begin{abstract}  
We prove a general existence result in stochastic optimal control in discrete time where controls take values in conditional metric spaces, and depend on the current state and the information of past decisions through the evolution of a recursively defined forward process. 
The generality of the problem lies beyond the scope of standard techniques in stochastic control theory such as random sets, normal integrands and measurable selection theory.  
The main novelty is a formalization in conditional metric space and the use of techniques in conditional analysis.  
We illustrate the existence result by several examples including wealth-dependent utility maximization under risk constraints with bounded and unbounded wealth-dependent control sets, utility maximization with a measurable dimension, and dynamic risk sharing. 
Finally, we discuss how conditional analysis relates to random set theory. 
\end{abstract}
\maketitle

\section{Introduction}
	The present work investigates parameter-dependent stochastic optimization in finite discrete time with the tools of conditional analysis. 
	In the following, we introduce the mathematical problem and sketch our solution strategy.  
	Given a forward generator $(v_t)_{t=0}^{T-1}$, consider a forward process 
	\[
	x_{t+1}=v_{t}(x_t,z_t)
	\]
	the dynamics of which depend on a parameter $x_t$ as a function of earlier decisions and an immediate decision $z_t$ chosen recursively in a state-dependent control set $\Theta_t(x_t)$ for each $t=0,\ldots,T-1$. 
	Given a backward generator $(u_t)_{t=0}^T$, the goal is to maximize 
	\begin{equation}\label{eq:global}
	u_0(x_0,\cdot,z_0)\circ\ldots\circ u_{T-1}(x_{T-1},\cdot,z_{T-1})\circ u_T (x_T) 
	\end{equation}
	over all possible forward processes initialized at $x_0$ where $\circ$ denotes composition of functions. 
	By the Bellman principle, the global stochastic optimization problem \eqref{eq:global} is solved by a backward recursion if all the local one-period problems
	\begin{align}
	\quad\quad\quad\quad y_t(x_t)&=\sup_{z_t\in\Theta_t(x_t)} u_t(x_t,y_{t+1}(v_t(x_t,z_t)),z_t), \quad t=0,\ldots,T-1, \label{prob:oneperiod}\\
	\quad\quad\quad\quad y_T(x_T)&=u_T(x_T) \nonumber
	\end{align}
	attain their maxima. 
	
	Given a filtered probability space $(\Omega,\mathcal{F},(\mathcal{F}_t)_{t=0}^T,\mathbb{P})$, we assume that the forward process $x_t$ and the control process $z_t$ assume values in $\mathcal{F}_t$-conditional metric spaces $X_t$ and $Z_t$ respectively. 	An $\mathcal{F}_t$-conditional metric space is a nonempty set $X$ endowed with a vector-valued metric $d:X\times X\to L^0_+(\Omega,\mathcal{F}_t,\mathbb{P})$ satisfying a concatenation property which encodes information at time $t$. An example is the space of strongly $\mathcal{F}_t$-measurable functions with values in a metric space with almost everywhere evaluation of the metric. Intuitively speaking, a conditional metric space is a collection of classical metric spaces $X(\omega)$, $\omega\in\Omega$, which are glued together in a measurable way. Instead of arguing in each $X(\omega)$ separately and building on measurable selection lemmas, we directly work in the conditional metric space $X$ and build instead on arguments in conditional analysis. 
	This is possible since in conditional metric spaces all basic results from real metric spaces are true in conditional form, cf.~\cite{carl2018transfer,cheridito2015conditional,drapeau2016algebra}. 
	
	Our main Theorem \ref{th:existence} shows that the global supremum in \eqref{eq:global} is attained and can be reduced  by Bellman's principle to the  
	local optimization problems \eqref{prob:oneperiod}. By backward induction, we show that the value function $y_t$ is upper semi-continuous on the conditional metric space $X_t$. For this, we assume that the control set $\Theta_t(x_t)$ is conditionally sequentially compact (for a discussion of the notion of conditional compactness, we refer to \cite[Sections 3 and 4]{cheridito2015conditional} and \cite[Sections 3.4 and 4]{drapeau2016algebra}). 
	 Then the existence of an optimizer in \eqref{prob:oneperiod} follows from a conditional version of the fact that a semi-continuous function on a compact space attains its extrema. 
	 Moreover, under a regularity condition on the control set - a conditional version of outer semi-continuity in set convergence (see e.g.~\cite[Chapter 5, Section B]{rockafellar02}) - it is shown that $y_t$ in \eqref{prob:oneperiod} is upper semi-continuous on $X_t$. The assumption
	 of conditional compactness on the control set is relaxed in Proposition \ref{sl} under stronger assumptions on the generators by modifying arguments in \cite{cheridito2016equilibrium}. In particular, we additionally require that the backward generators $u_t$ are $\mathcal{F}_t$-sensitive to large losses and increasing in the state variable. 
	 
	 In Section \ref{sec:compcond} we provide sufficient conditions for conditional compactness and conditional outer semi-continuity of the control set. We focus on conditionally finite dimensional control sets. The results are illustrated with  applications in mathematical finance. In Example \ref{ex:riskconstarints} we study  an optimal consumption problem with local risk constraints on the wealth process.  
	 Example \ref{ex:measurabledim} indicates the importance of conditional Euclidean space with measurable dimension to model control processes with state-dependent dimension (e.g.~the number of traded assets at time $t$ depends on $\mathcal{F}_t$ and past decisions).
	 As an application of Proposition \ref{sl}
	 we derive optimal portfolios w.r.t.~dynamic risk measures for which the risk aversion coefficient is influenced by the current wealth. 
	 Moreover,  a closed-form solution to a dynamic wealth-dependent risk sharing problem is obtained which extends the formula of Borch \cite{borch1962}.

Normal integrands are a widely used tool to investigate parametrized stochastic optimization, see e.g.~\cite{biaginiduality,pennanen2011convex,rockafellar1974,rockwets78} and \cite{rockafellar1976integral} for an introduction. 
In Section \ref{sec:randomset}, we establish a connection between conditional analysis and random sets, normal integrands and measurable selection theory. 
In Theorem \ref{t:selec}, we prove a one-to-one correspondence between the set of measurable selections of Effros measurable and closed-valued mappings and stable and sequentially closed sets. 
This result yields a one-to-one correspondence between normal integrands and stable and sequentially semi-continuous functions. 
This indicates that control problems formulated in the language of normal integrands and random sets can equally be formulated in the language of conditional analysis.  
For a formalization with normal integrands and random sets, measurable selection lemmas provide the main tool to secure measurability.  
The use of measurable selection arguments is enforced by a pointwise application of standard results in classical analysis, and relies on topological assumptions such as separability and standard Borel spaces.  
In this regard, conditional analysis provides a measure-theoretic alternative which does not rely on any topological assumptions, and works as soon as a formalization within its language is reached which is demonstrated in this article in discrete time stochastic control theory.   
Conditional analysis approaches measurable functions directly by providing a measurable (or conditional, or stochastic, or random) version of results in classical analysis. 
The application of conditional versions of classical theorems preserves measurability, see for example the proofs below in which a measurable  version of the Bolzano-Weierstra\ss~theorem, the maximum theorem and the Heine-Borel theorem are employed.  
This perspective is implicitly present in \cite{hansen1987role,kabanov2001teachers}, however without a systematic treatment. 
A conditional version of basic results in functional analysis were established in \cite{cheridito2015conditional,filipovic2009separation,drapeau2016algebra}, and applied to financial mathematics in \cite{backhoff2016conditional,bielecki2016dynamic,cheridito2016equilibrium,cheridito2011optimal,drapeau2016conditional,filipovic2012approaches,frittelli2011conditional,frittelli2014complete,zapata16}.   
In \cite{drapeau2016algebra}, conditional versions of classical theorems were studied systematically and related to a conditional variant of set theory.  
This naturally raises connections with mathematical logic, see \cite{bell2005set,carl2018transfer,maclane2012sheaves} for related literature. 
Conditional analysis in $L^0$-modules is moreover investigated in e.g.~\cite{CerreiaVioglio2017,jamneshan2017compact,orihuela2017stability}.  
Related results in randomly normed modules are studied in e.g.~\cite{guo2010relations,guoshi11,haydon1991randomly}. \\

The remainder of this article is organized as follows. In Section \ref{sec1} we introduce the notion of conditional metric spaces and prove the main existence result. In Section \ref{sec:compcond} and Section \ref{sec:unbounded} we discuss extensions of the main result and provide several examples. The link between conditional analysis and random set theory is established in Section \ref{sec:randomset}.

\section{Main result}\label{sec1}
Let $(\Omega,\mathcal{F},\mathbb{P})$ be a probability space. Throughout we identify two sets in $\mathcal{F}$ whenever their symmetric difference is a null set, and identify two functions on $\Omega$ if they coincide a.s.~(almost surely). 
Let $\mathcal{G}$ be a sub-$\sigma$-algebra of $\mathcal{F}$.  
Denote by $\Pi_\mathcal{G}$ the set of partitions $(A_k)$ of $\Omega$ where $A_k\in \mathcal{G}$ for all $k$.  
Let $L^0_\mathcal{G}$, $L^0_\mathcal{G}(\mathbb{N})$, $L^0_{\mathcal{G},+}$, $L^0_{\mathcal{G},++}$, $\underbar{L}^0_\mathcal{G}$, and $\bar{L}^0_\mathcal{G}$ denote the spaces of $\mathcal{G}$-measurable random variables with values in $\mathbb{R}$, $\mathbb{N}$, $[0,\infty)$, $(0,\infty)$, $\mathbb{R}\cup\{-\infty\}$ and $\mathbb{R}\cup\{\pm\infty\}$ respectively. 
Recall that $L^0_\mathcal{G}$  with the pointwise a.s.~order is a Dedekind complete lattice-ordered ring. The essential supremum and the essential infimum are denoted by $\sup$ and $\inf$ respectively. 
Inequalities between random variables with values in an ordered set are always understood in the pointwise a.s.~sense. 

\begin{definition}\label{def:condMetric}
	A $\mathcal{G}$-\emph{conditional metric} on a non-empty set $X$ is a function $d:X\times X\rightarrow L^0_{\mathcal{G},+}$ such that the following conditions hold: 
	\begin{itemize}
		\item[(i)]  $d(x,y)=0$ if and only if $x=y$,
		\item[(ii)]  $d(x,y)=d(y,x),$ 
		\item[(iii)] $d(x,z)\leq d(x,y)+d(y,z)$, 
		\item[(iv)] for every sequence $(x_k)$ in $X$ and $(A_k)\in \Pi_\mathcal{G}$ there exists exactly one element $x\in X$ such that $1_{A_k}d(x,x_k)=0$ for all $k\in\mathbb{N}$.
	\end{itemize}
	The pair $(X,d)$ is called a $\mathcal{G}$-\emph{conditional metric space}.
\end{definition}
In the following we call the unique element in (iv) the \emph{concatenation} of the sequence $(x_k)$ along the partition $(A_k)$ and denote it by $\sum_k 1_{A_k}x_k$. 
For a sequence $(x_n)$ in a conditional metric space $(X,d)$ we write $x_n\to x$ a.s. whenever $d(x,x_n)\to 0$ a.s. A \emph{measurable subsequence} $(x_{n_k})$ of $(x_n)$ is of the form $x_{n_k}:=\sum_{j\in\mathbb{N}} 1_{\{n_k=j\}}x_j$ where $(n_k)$ is a sequence in $L^0_{\mathcal{G}}(\mathbb{N})$ such that $n_k<n_{k+1}$ for all $k\in\mathbb{N}$.  

\begin{definition}\label{def:basic}
	Let $(X,d_X)$ and $(Z,d_Z)$ be  $\mathcal{G}$-\emph{conditional metric spaces}, and $H$ and $G$ subsets of $X$ and $Z$, respectively. 
	We call $H$
	\begin{itemize}
		\item \emph{$\mathcal{G}$-stable} if $H \neq \emptyset$ and $\sum_k 1_{A_k} x_k\in H$ for all $(A_k)\in \Pi_\mathcal{G}$ and every sequence $(x_k)$ in $H$,
		\item \emph{sequentially closed} if $H$ contains every $x\in X$ such that there is a sequence $(x_k)$ in $H$ with $x_k\to x$ a.s.
	\end{itemize}
	
	A function $f\colon H\to G$ is said to be 
	\begin{itemize}
		\item \emph{$\mathcal{G}$-stable} if $f\left(\sum_k 1_{A_k} x_k\right)=\sum_k 1_{A_k} f(x_k)$ for all $(A_k)\in \Pi_\mathcal{G}$ and every sequence $(x_k)$ in $H$, where $H$ and $G$ are assumed to be $\mathcal{G}$-stable,
		\item \emph{sequentially continuous} whenever $f(x)=\lim_k f(x_k)$ if $x_k\to x$ a.s.~in $H$,
	\end{itemize}
	and if $G=\bar{L}^0_\mathcal{F}$, then $f$ is said to be 
	\begin{itemize}
		\item \emph{sequentially lower semi-continuous} if $f(x)\leq \liminf f(x_k)$ if $x_k\to x$ a.s.~in $H$,
		\item \emph{sequentially upper semi-continuous} if $\limsup f(x_k)\leq f(x)$ if $x_k\to x$ a.s.~in $H$.
	\end{itemize}
\end{definition}

\begin{remark}\label{rem: Stmetric}
	{\bf 1.} If $(X,d)$ is a $\mathcal{G}$-conditional metric space then the metric $d:X\times X\rightarrow L^0_{\mathcal{G},+}$ is $\mathcal{G}$-stable, i.e. $d\left(\sum_k 1_{A_k}x_k, \sum_k 1_{A_k}y_k\right)=\sum_k 1_{A_k}d(x_k,y_k)$ for every sequences $(x_k)$ and $(y_k)$ in $X$ and $(A_k)\in \Pi_\mathcal{G}$. Indeed, denoting by
	$x=\sum_k 1_{A_k}x_k$ and $y=\sum_k 1_{A_k}y_k$ the respective
	concatenations, it follows from the triangular inequality that
	\begin{align*}
	1_{A_k}d(x,y)&\le 1_{A_k}d(x,x_k)+1_{A_k}d(x_k,y_k)+1_{A_k}d(y_k,y)=1_{A_k}d(x_k,y_k)\\
	&\le 1_{A_k}d(x_k,x)+1_{A_k}d(x,y)+1_{A_k}d(y_k,y)=1_{A_k}d(x,y)
	\end{align*}
	which shows that $1_{A_k}d(x,y)=1_{A_k}d(\sum_k 1_{A_k}x_k,\sum_k 1_{A_k}y_k)=1_{A_k}d(x_k,y_k)$ for all $k\in\mathbb{N}$. Summing up over all $k$ yields the desired $\mathcal{G}$-stability.
	
	{\bf 2.} Let $(X,d_X)$ and $(Y,d_Y)$ be two $\mathcal{G}$-conditional metric spaces. Then its product $X\times Y$ endowed with the $\mathcal{G}$-conditional metric $d_{X\times Y}((x,y),(x^\prime,y^\prime)=\max\{d_{X}(x,x^\prime),d_{Y}(y,y^\prime)\}$ is a $\mathcal{G}$-conditional metric space.
	In the following all products of conditional metric spaces are endowed with this conditional metric.   
	
	{\bf 3.} Let $(X,d_X)$ be a $\mathcal{G}$-conditional metric spaces. Then the set $\mathbf{X}$ of all pairs $(x,A)\in X\times\mathcal{G}$, where $(x,A)$ and $(y,B)$ are identified if $A=B$ and $1_A d(x,y)=0$ is a \emph{conditional set}. In general, a conditional set $\mathbf{Y}$ is an abstraction
	of this example, and can be viewed as a set-like structure on which $\mathcal{G}$ acts such that $\mathbf{Y}$ is closed w.r.t.~countable concatenations of its elements along partitions in $\Pi_{\mathcal{G}}$. Conditional set theory
	is investigated in \cite{drapeau2016algebra} and does not require a metric structure as in Definition \ref{def:condMetric}. For further results on conditional metric spaces in the context of conditional set theory we refer to \cite[Section 4]{drapeau2016algebra}.    
\end{remark}

We next introduce the parameter-dependent stochastic optimal control problem for conditional metric spaces. For a fixed finite time horizon $T\in\mathbb{N}$,
we consider a filtration $\mathcal{F}_0\subset \mathcal{F}_1\subset  \ldots \subset  \mathcal{F}_T=\mathcal{F}$. For simplicity, we often abbreviate the index $\mathcal{F}_t$ by $t$, and write for instance $L^0_{t}$ for $L^0_{\mathcal{F}_t}$.
For each $t=0,\ldots,T$, let $(X_t,d_{X_t})$ and $(Z_t,d_{Y_t})$ be $\mathcal{F}_t$-conditional metric spaces. 
Our aim is to study control problems for which the control set $\Theta_t$ depends on $\mathcal{F}_t$, but also on a state parameter $x\in X_t$. For every $t=0,\ldots,T-1$,  we assume that the \emph{state-dependent control set}
$\Theta_t$ satisfies   
\begin{enumerate}
	\item[(c1)] $\emptyset\neq\Theta_t(x)\subset  Z_t$ for all $x\in X_t$,
	\item[(c2)] $\Theta_t$ is $\mathcal{F}_t$-stable, i.e. 
	\[
	\Theta_t\Big(\sum_k 1_{A_k}x_k\Big)=\sum_k 1_{A_k} \Theta_t(x_k):=\Big\{\sum_k 1_{A_k} z_k\colon z_k\in\Theta_t(x_k) \text{ for all }k\Big\}
	\]
	for all $(A_k)\in \Pi_t$ and every sequence $(x_k)$ in $X_t$,
	\item[(c3)] for every $x\in X_t$, the set $\Theta_t(x)$ is \emph{conditionally sequentially compact}, i.e.~for every sequence $(z_n)$ in $\Theta_t(x)$ there exists a measurable subsequence $n_1<n_2<\cdots$ with $n_k\in L^0_t(\mathbb{N})$ such that $z_{n_k}\to z\in \Theta_t(x)$ a.s., 
	\item[(c4)] 
	for every sequence $(x_n)$ in $X_t$ such that $x_n\to x\in X_t$ a.s.~and every sequence $(z_n)$ in $\Theta_t(x_n)$ there exists a measurable subsequence $n_1<n_2<\cdots$ with $n_k\in L^0_t(\mathbb{N})$  	
	and a sequence $(z_k^\prime)$ in $\Theta_t(x)$ such that 
	$d_{Z_t}(z_{n_k},z_{k}^\prime)\to 0$ a.s. 
\end{enumerate}
Note that $\mathcal{F}_t$-stability of $\Theta_t$ implies that $\Theta_t(x)$ is $\mathcal{F}_t$-stable for all
$x\in X_t$.

We consider \emph{forward generators}
\[
v_t:X_t\times Z_t\rightarrow X_{t+1},\quad t=0,\ldots,T-1, 
\]
which are
\begin{enumerate}[label=(v\arabic*)]
	\item\label{v1} $\mathcal{F}_t$-stable, 
	\item\label{v2} sequentially continuous. 
\end{enumerate}
For every $x_t\in X_t$ we consider the set
\[
C_t(x_t):=\left\{\big((x_s)_{s=t+1}^T,(z_s)_{s=t}^{T-1}\big)
\colon x_{s+1}=v_s(x_s,z_s),z_s\in\Theta_s(x_s)\textnormal{ for all }s=t,\ldots,T-1\right\}
\]   
of all parameter processes $(x_s)_{s=t}^T$ which can be realized by the state-dependent controls $z_s\in\Theta_t(x_s)$
for $s=t,\dots,T-1$. 

As for the objective function, we consider \emph{backward generators} 
\[
u_t:X_t\times\underbar{L}^0_{t+1}\times Z_t \rightarrow\underbar{L}^0_t,\quad t=0,1,\ldots,T-1, 
\]
which are
\begin{enumerate}[label=(u\arabic*)]
	\item\label{u1} $\mathcal{F}_t$-stable, 
	\item\label{u2} increasing in the second component, 	
	\item\label{u3} sequentially upper semi-continuous.  
\end{enumerate}
We assume that $u_T: X_T\rightarrow\underbar{L}^0_T$ is $\mathcal{F}_{T}$-stable and sequentially upper semi-continuous. Given such a family $(u_t)_{t=0}^T$ of backward generators, our goal is to maximize 
\begin{equation}\label{opt:global}
y_t(x_t):=\sup_{((x_s)_{s=t+1}^T,(z_s)_{s=t}^{T-1})\in C_t(x_t) }u_t(x_t,\cdot,z_t)\circ\cdots\circ u_{T-1}(x_{T-1},\cdot,z_{T-1})\circ u_T (x_T) 
\end{equation}
over all realizable state processes initialized at $x_t\in X_t$.  
\begin{remark}
	The objective function in the stochastic control problem \eqref{opt:global} is recursively defined. 
	Its generators (aggregators) are functions between conditional metric spaces which satisfy monotonicity and semi-continuity. The aggregators are not necessarily (conditional) expected utilities. 
	In case of (conditional) expected utility, the generators are closely related with dynamic and conditional risk measures, see \cite{acciaio2011dynamic,bielecki2016dynamic,detlefsen2005conditional,filipovic2012approaches,frittelli2011conditional}. 
	The preferences which underly conditional expected utility functionals were studied in \cite{drapeau2016conditional} under the name of conditional preference orders.
	
	In decision theory, there is an extensive literature on recursive utilities starting with the seminal work \cite{kreps1978temporal,kreps1979temporal}. 
	The preferences therein are defined on sets of temporal lotteries (probability trees), and follow a kind of Bellman recursive structure which is  similar to the construction above on a formal level (see \cite[Theorem 1]{kreps1978temporal}). 
	This was later extended in \cite{epstein1989substitution} where non-expected utilities were incorporated as well, and established under the name of Epstein-Zin utilities.  
	See also  \cite{machina2008nonexpected} for a survey on non-expected utility theory. 
	With the techniques of conditional analysis and based on results in BSDE theory, \cite{cheridito2011optimal} solves a utility maximization problem in continuous time for Epstein-Zin utilities. 
\end{remark}

The following result shows that the global supremum in \eqref{opt:global} is attained and can be reduced to local optimization problems by the following Bellman's principle.
\begin{theorem}
	\label{th:existence}
	Suppose that (c1)--(c4), \ref{v1}--\ref{v2},  and \ref{u1}--\ref{u3} are fulfilled. Then the functions $y_t\colon X_{t}\rightarrow\underbar{L}^0_t$ are $\mathcal{F}_t$-stable and sequentially upper semi-continuous for all $t=0,\dots, T$, and can be computed by backward recursion	
	\begin{align*}
	y_T(x_T)&=u_T(x_T)\\
	y_t(x_t)&=\max_{z_t\in\Theta_t(x_t)} u_t(x_t,y_{t+1}(v_t(x_t,z_t)),z_t), \quad t=0,\ldots,T-1.
	\end{align*}
	Moreover, for every $x_t\in X_t$ the process
	$((x^\ast_s)_{s=t}^T,(z^\ast_s)_{s=t}^{T-1})$ given by
	$x_t^\ast=x_t$ and the forward recursion
	\begin{equation}\label{eq:rec}
	x^\ast_{s+1}=v_s(x^\ast_s,z^\ast_s)\quad\mbox{where}\quad z_s^\ast\in \mathop{\rm argmax}_{z_s\in\Theta_s(x_s^\ast)} u_s\big(x_s^\ast,y_{s+1}(v_t(x^\ast_s,z_s)),z_s\big),\quad s=t,\dots T-1,
	\end{equation} 
	satisfies $((x^\ast_s)_{s=t+1}^T,(z^\ast_s)_{s=t}^{T-1})\in C_t(x_t)$ and 	
	\[
	y_t(x_t)= u_t(x_t,\cdot,z^\ast_t)\circ\cdots\circ u_{T-1}(x^\ast_{T-1},\cdot,z^\ast_{T-1})\circ u_T (x^\ast_T). 
	\]	
\end{theorem}
\begin{proof}
	The proof is by backward induction. For $t=T$ it follows from \eqref{opt:global} that $y_T=u_T$ which by assumption is an $\mathcal{F}_t$-stable and sequentially upper semi-continuous function from $X_T$ to $\underbar{L}^0_T$.
	
	As for the induction step, assume that $y_{t+1}\colon X_{t+1}\rightarrow\underbar{L}^0_{t+1}$
	is $\mathcal{F}_{t+1}$-stable
	and sequentially upper semi-continuous, and that for each $x_{t+1}\in X_{t+1}$ there exists $((x^\ast_s)_{s=t+2}^T,(z^\ast_s)_{s=t+1}^{T-1})\in C_{t+1}(x_{t+1})$ such that $y_{t+1}(x_{t+1})=u_{t+1}(x_{t+1},\cdot,z^\ast_{t+1})\circ\cdots\circ u_T (x^\ast_T)$. By \ref{u1} and \ref{v1} the function
	\[
	X_t\times Z_t\ni(x,z)\mapsto u_t\big(x,y_{t+1}(v_t(x,z)),z\big)
	\]
	is $\mathcal{F}_t$-stable. Moreover, it is sequentially upper semi-continuous. Indeed, let $(x_k,z_k)$ be a sequence in $X_t\times Z_t$ such that $x_k\to x\in X_t$ a.s. and $z_k\to z\in Z_t$ a.s. Since
	$v(x_k,z_k)\to v(x,z)$ a.s. by \ref{v2} it follows from the induction hypothesis that 
	\[
	\limsup_{k\to\infty} y_{t+1}(v_t(x_k,z_k))\le y_{t+1}(v(x,z))<+\infty. 
	\]
	Since 
	\[
	\big\{ \sup_{k\ge 1} y_{t+1}(v_t(x_{k},z_{k}))=+\infty \big\}
	=\bigcap_{k\ge 1}\big\{ \sup_{k^\prime\ge k} y_{t+1}(v_t(x_{k^\prime},z_{k^\prime}))=+\infty \big\}
	=\big\{\limsup_{k\to\infty} y_{t+1}(v_t(x_k,z_k))=+\infty\big\}, 
	\]
	we have $\sup_{k\ge 1} y_{t+1}(v_t(x_{k},z_{k}))\in \underbar{L}^0_{t+1}$. Hence, by \ref{u2}, \ref{u3} and \ref{v2}  we get
	\begin{align}
	\limsup_{k\to\infty} u_t\big(x_k,y_{t+1}(v_t(x_k,z_k)),z_k\big)
	&\le \limsup_{k\to\infty} u_t\big(x_k,\sup_{k^\prime\ge k} y_{t+1}(v_t(x_{k^\prime},z_{k^\prime})),z_k\big) \nonumber\\
	&\le u_t\big(x,\limsup_{k\to\infty} y_{t+1}(v_t(x_{k},z_{k})),z\big) \nonumber\\
	&\le u_t\big(x, y_{t+1}(v_t(x,z)),z\big)\label{eq:usc}
	\end{align}
	which shows the desired sequential upper semi-continuity.
	As a consequence, the supremum in
	\begin{equation}\label{eq:ft}
	f_t(x_t):=\sup_{z\in\Theta_t(x_t)} u_t\big(x_t,y_{t+1}(v_t(x_t,z)),z\big)
	\end{equation}
	is attained for each $x_t\in X_t$. Indeed, since $z\mapsto u_t\big(x,y_{t+1}(v_t(x,z)),z\big)$ and $\Theta_t(x_t)$ are $\mathcal{F}_t$-stable, it follows from standard properties of the essential supremum that there exists a sequence $z_n\in\Theta_t(x_t)$ such that 
	\[
	u_t\big(x_t,y_{t+1}(v_t(x_t,z_n)),z_n\big)\to f_t(x_t)\quad\mbox{a.s.}
	\]
	By (c3)  there is a measurable subsequence $n_1<n_2<\cdots$ with $n_k\in L^0_t(\mathbb{N})$ such that $z_{n_k}\to z\in \Theta_t(x_t)$ a.s. Since $z\mapsto u_t\big(x,y_{t+1}(v_t(x,z)),z\big)$ is sequentially upper semi-continuous and $\mathcal{F}_t$-stable, it follows that 
	\[
	u_t\big(x_t,y_{t+1}(v_t(x_t,z)),z\big)\ge
	\limsup_{k\to\infty}
	u_t\big(x_t,y_{t+1}(v_t(x_t,z_{n_k})),z_{n_k}\big) = f_t(x_t)
	\]
	which shows that the supremum in \eqref{eq:ft} is attained. 
	
	We next show that $f_t:X_t\to \underbar{L}^0_{t+1}$ is sequentially upper semi-continuous. 
	By contradiction, suppose that $(x_k)$ is a sequence in $X_t$ such that $x_k\to x\in X_t$ a.s. and   $f_t(x)<\limsup_k f_t(x_k)$ on some $A\in\mathcal{F}$ with $\mathbb{P}(A)>0$. 
	Note that $f_t$ is $\mathcal{F}_t$-stable. Thus, by possibly passing to a measurable subsequence, we can suppose that there exists $r\in L^0_{t,++}$ such that
	\begin{equation}
	\label{eq: separation}
	f_t(x)+r< f_t(x_k)\text{ on }A,\quad\text{ for all }k\in\mathbb{N}.
	\end{equation}
	
	Denote by $z_k\in\Theta_t(x_k)$ a respective maximizer of $f_t(x_k)$. By (c4) there exists $z_k^\prime\in\Theta_t(x)$ such that $d_{Z_t}(z_k,z^\prime_k)\to 0$ a.s.~by possibly passing to a measurable subsequence. By (c3) there exists a measurable subsequence
	$k_1<k_2<\cdots$ with $k_l\in L^0_t(\mathbb{N})$ such that $z^\prime_{k_l}\to z^\prime\in\Theta_t(x)$ a.s. Since $d_{Z_t}(z_{k_l},z^\prime_{k_l})\to 0$ a.s.~by $\mathcal{F}_t$-stability of the conditional metric $d_{Z_t}$, it follows from the triangular inequality that $z_{k_l}\to z^\prime\in\Theta_t(x)$ a.s. By the $\mathcal{F}_t$-stability of $f_t$ and (c2) it follows that $z_{k_l}$ is in $\Theta_t(x_{k_l})$ and maximizes $f_t(x_{k_l})$. Hence, it follows from \eqref{eq:usc} that
	\begin{align*}
	\limsup_{l\to\infty} f_t(x_{k_l})
	&=\limsup_{l\to\infty} u_t\big(x_{k_l},y_{t+1}(v_t(x_{k_l},z_{k_l})),z_{k_l}\big)\\
	&\le u_t\big(x,y_{t+1}(v_t(x,z^\prime)),z^\prime\big)\\
	&\le \sup_{z\in\Theta_t(x)} u_t\big(x,y_{t+1}(v_t(x,z),z\big)=f_t(x).
	\end{align*}
	Notice that, due to the $\mathcal{F}_t$-stability of $f_t$, (\ref{eq: separation}) is satisfied for any measurable subsequence of $(x_k)$. 
	Thus, we have that $f_t(x)+r\le\limsup_{l\to\infty} f_t(x_{k_l})\le f_t(x)$ on $A$, which is a contradiction. 
	We conclude that $f_t$ is sequentially upper semi-continuous.
	
	Finally, we show that $y_t=f_t$. By induction hypothesis, for every $x_t\in X_t$ and $z_t\in Z_t$ there exists $\big((x^\ast_s)_{s=t+2}^T,(z^\ast_s)_{s=t+1}^{T-1}  \big)\in C_{t+1}(v_t(x_t,z_t))$ such that 
	\[
	y_{t+1}(v_t(x_t,z_t))=u_{t+1}(v_t(x_t,z_t),\cdot,z^\ast_{t+1})\circ\cdots\circ u_{T-1}(x^\ast_{T-1},\cdot,z^\ast_{T-1})\circ u_T(z^\ast_T).
	\]
	In particular, for $x_t\in X_t$ and $z_t^\ast\in Z_t$ being a respective maximizer in \eqref{eq:ft} one has
	\begin{align*}
	f_t(x_t)&=\sup_{z\in\Theta_t(x_t)} u_t\big(x_t,y_{t+1}(v_t(x_t,z)),z\big)\\
	&= u_t\big(x_t,y_{t+1}(v_t(x_t,z_t^\ast)),z_t^\ast\big)\\
	&= u_t\big(x_t,\cdot,z_t^\ast\big)\circ u_{t+1}(v_t(x_t,z^\ast_t),\cdot,z^\ast_{t+1})\circ\cdots\circ
	u_{T-1}(x^\ast_{T-1},\cdot,z^\ast_{T-1})\circ u_T(z^\ast_T) \\
	&= \sup_{((x_s)_{s=t+2}^T,(z_s)_{s=t+1}^T)\in C_{t+1}(v(x_t,z^\ast_t)) }
	u_t\big(x_t,\cdot,z_t^\ast\big)\circ u_{t+1}(v_t(x_t,z_t),\cdot,z_{t+1})\circ\cdots\circ\circ u_T(z_T) \\
	&= \sup_{z_t\in\Theta_t(x_t)}
	\sup_{((x_s)_{s=t+2}^T,(z_s)_{s=t+1}^T)\in C_{t+1}(v(x_t,z_t)) }
	u_t\big(x_t,\cdot,z_t\big)\circ u_{t+1}(v_t(x_t,z_t),\cdot,z_{t+1})\circ\cdots\circ u_T(z_T) \\
	&= \sup_{((x_s)_{s=t+1}^T,(z_s)_{s=t}^T)\in C_{t}(x_t) }
	u_t\big(x_t,\cdot,z_t\big)\circ u_{t+1}(v_t(x_t,z_t),\cdot,z_{t+1})\circ\cdots\circ\circ u_T(z_T) \\
	&=y_t(x_t).
	\end{align*}
	This shows that
	$((x^\ast_s)_{s=t+1}^T,(z^\ast_s)_{s=t}^T)\in C_{t}(x_t)$
	is an optimizer of \eqref{opt:global} whenever it
	satisfies the local optimality criterion
	\[
	z^\ast_s\in\mathop{\rm argmax}_{z\in\Theta_t(x^\ast_t)} u_s\big(x_s^\ast,y_{s+1}(v_t(x^\ast_s,z_s)),z_s\big)\quad\mbox{and}\quad x^\ast_{s+1}=v_s(x^\ast_s,z^\ast_s)
	\]
	for all $s=t,\dots,T$, where $x_t^\ast=x_t$. In particular, every process which satisfies the forward recursion \eqref{eq:rec} is an optimizer for \eqref{opt:global}.
\end{proof}	

\begin{examples}\label{rem:extensions} 
	As for the illustration we provide examples of $\mathcal{F}_t$-conditional metric spaces which are of interest for the control and parameter spaces in Theorem \ref{th:existence}.
	
	{\bf 1.} Given a nonempty metric space $(X,d)$, denote by $L^0_t(X)$ the set of all strongly $\mathcal{F}_t$-measurable functions $x\colon \Omega\to X$,
	i.e.~the set of those $x$ for which there exists a sequence $(x^n)$ of countable simple functions $x^n=\sum_{k} 1_{A^n_k} x^n_k $  with $x^n_k\in X$ and $(A^n_k)\in\Pi_t$, such that
	$d(x(\omega),x^n(\omega))\to 0$ for a.a.~$\omega\in\Omega$.
	Notice that the metric $d$ extends from $X$ to $L^0_t(X)$ with values in $L^0_{t,+}$ by defining 
	\[d_{L^0_t(X)}(x,\bar x):=\lim_{n\to\infty} d(x^n,\bar x^n)\] where 
	$x^n=\sum_{k}  1_{A^n_k} x^n_k$ and  $\bar x^n=\sum_{k}  1_{\bar A^n_k} \bar x^n_k$ are  sequences of countable simple functions such that $d(x(\omega),x^n(\omega))\to 0$ and 
	$d(\bar x(\omega),\bar x^n(\omega))\to 0$ for a.a.~$\omega\in\Omega$, and   
	\[d(x^n,\bar x^n):=\sum_{k,k^\prime} 1_{A_k^n\cap \bar A^n_{k^\prime}} d(x^n_k,\bar x^n_{k^\prime}).\]
	Notice that $d_{L^0_t(X)}(x,\bar x)$ does not depend on the choice of approximating sequences $(x^n)$ and $(\bar x^n)$. 
	Then $(L^0_t(X),d_{L^0_t(X)})$ is a $\mathcal{F}_t$-conditional metric space. For instance,
	if $(x_k)$ is a sequence in $L^0_t(X)$ such that $x_k$ is the limit of the countable sequence $(x^n_k)$ and $(A_k)\in\Pi_t$ then the concatenation $\sum_k  1_{A_k} x_k$ is the unique element in $L^0_t(X)$ given as the limit of
	the countable simple functions $\sum_k 1_{A_k} x^n_k $ for $n\to\infty$. 
	
	{\bf 2.} The conditional Euclidean space with measurable dimension  $n=\sum_k  1_{A_k} n_k\in L^0_t(\mathbb{N})$ is defined as 
	\[
	L^0_t(\mathbb{R})^n=\sum_k  1_{A_k}  L^0_t(\mathbb{R}^{n_k}):=\Big\{\sum_k  1_{A_k} x_k  \colon x_k\in L^0_t(\mathbb{R}^{n_k}) \text{ for all }k\Big\}. 
	\]
	The $\mathcal{F}_t$-conditional metric on $L^0_t(\mathbb{R})^n$ is defined by
	\[ d_{L^0_t(\mathbb{R})^n}(x,\bar x):=\sum_k 1_{A_k} d_{ L^0_t(\mathbb{R}^{n_k}) }(x_k,\bar x_k), \] 
	where $x=\sum_k  1_{A_k} x_k$ and $\bar x=\sum_k  1_{A_k} \bar x_k$. Here, $d_{L^0_t(\mathbb{R}^{n_k})}$ denotes the $\mathcal{F}_t$-conditional metric on $L^0_t(\mathbb{R}^{n_k})$ which extends the Euclidean metric on $\mathbb{R}^{n_k}$ as defined in the previous example.
	Straightforward verification shows that $(L^0_t(\mathbb{R})^n, d_{L^0_t(\mathbb{R})^n})$ is a $\mathcal{F}_t$-conditional metric space.
	
	{\bf 3.} For $1\leq p<\infty$, we define the conditional $L^p$-space 
	\[
	L^p_t:=\{x\in L^0_T \colon \mathbb{E}[|x|^p|\mathcal{F}_t]<+\infty \text{ a.s.}\}  
	\]
	with $\mathcal{F}_t$-conditional metric
	$d_{L^p_t}(x,\bar x):=\mathbb{E}[|x-\bar x|^p|\mathcal{F}_t]^{1/p}$. By definition, $(L^p_t,d_{L^p_t})$ is a $\mathcal{F}_t$-conditional metric space.
\end{examples}

\section{Compactness condition for the control set}\label{sec:compcond}

\subsection{The finite dimensional case}
Suppose that $Z_t=L^0_t(\mathbb{R}^d)$.  As shown in Example \ref{rem:extensions} the Euclidean metric of $\mathbb{R}^d$ extends to the $\mathcal{F}_t$-conditional metric $d_{L^0_t(\mathbb{R}^d)}:L^0_t(\mathbb{R}^d)\to L^0_{t,+}$.

\begin{proposition}\label{prop:controlfinite}
	Suppose that for each $t=0,\dots,T-1$, the control set $\Theta_t$ satisfies (c1), (c2) and the following conditions:
	\begin{itemize}
		\item[(i)] $\big\{(x,z)\in X_t\times L^0_t(\mathbb{R}^d):z\in\Theta_t(x)\big\}$ is sequentially closed, 
		\item[(ii)] for every  sequence $(x_n)$ in $X_t$
		with $x_n\to x\in X_t$ a.s.~there exists $M\in L^0_{t,+}$ such that $d_{L^0_t(\mathbb{R}^d)}(z,0)\le M$ for all $z\in\bigcup_n \Theta_t(x_n)$.
	\end{itemize}
	Then, the control set $\Theta_t$ satisfies (c1)-(c4).
\end{proposition}
\begin{proof}
	Let $(x_n)$ in $X_t$ be a sequence such that $x_n\to x\in X_t$ a.s., and $z_n\in \Theta_t(x_n)$. Since by assumption $d_{L^0_t(\mathbb{R}^d)}(z_n,0)\le M$ for some $M\in L^0_{t,+}$, the conditional Bolzano-Weierstrass theorem
	\cite[Theorem 3.8]{cheridito2015conditional}
	implies a measurable subsequence $n_1<n_2<\cdots$ with $n_k\in L^0_t(\mathbb{N})$ such that  $d_{L^0_t(\mathbb{R}^d)}(z_{n_k},z)\to 0$ a.s. for some $z\in L^0_t(\mathbb{R}^d)$. Since $\Theta_t$ satisfies (c2) one has $z_{n_k}\in\Theta_t(x_{n_k})$,
	and therefore $z\in \Theta_t(x)$ by (i).
	This shows (c4) and (c3) follows by considering the constant sequence $x_n=x$ for all $n\in\mathbb{N}$.
\end{proof}	


\begin{example}\label{ex:riskconstarints}
	Let $(S_t)_{t=0}^T$ be a $d$-dimensional $(\mathcal{F}_t)$-adapted price process. Given an initial investment $x_0>0$, we consider the wealth process
	\[
	x_{t+1}=v_t(x_t,z_t):=x_t+ \vartheta_t\cdot\Delta S_{t+1}-c_t,
	\]
	where the control $z_t=(\vartheta_t,c_t)\in L^0_t(\mathbb{R}^d)\times L^0_+$ consists of an investment strategy $\vartheta_t\in L^0_t(\mathbb{R^d})$ and a 
	consumption $c_t\in L^0_{t,+}$. The forward generator $v_t:L^0_t \times L^0_t(\mathbb{R}^d\times\mathbb{R}_+)\to L^0_{t+1}$
	satisfies \ref{v1} and \ref{v2}. We assume that the wealth process is regulated by
	\begin{equation}\label{eq:risk}
	\rho_t(x_{t+1})\le 0,
	\end{equation}
	i.e. $x_{t+1}$ is acceptable w.r.t.~a $\mathcal{F}_t$-conditional convex risk measure
	$\rho_t:L^0_{t+1}\to \bar{L}^0_t$ for all $t=0,\dots,T-1$. Recall that a $\mathcal{F}_t$-conditional convex risk measure is
	\begin{itemize}
		\item \emph{normalized}, i.e. $\rho_t(0)=0$,
		\item \emph{monotone}, i.e. $\rho_t(x)\le\rho_t(y)$ for all $x,y\in L^0_{t+1}$ with $x\ge y$,
		\item \emph{$\mathcal{F}_t$-translation invariant}, i.e. $\rho_t(x+m)=\rho(x)-m$ for all $x\in L^0_{t+1}$ and $m\in L^0_t$,
		\item \emph{$\mathcal{F}_t$-convex}, i.e. $\rho_t(\lambda x+(1-\lambda)y)\le\lambda\rho_t(x)+(1-\lambda)\rho_t(y)$ for all $x,y\in L^0_{t+1}$ and $\lambda\in L^0_t$ with $0\le \lambda\le 1$.
	\end{itemize}
	By $\mathcal{F}_t$-translation invariance it follows that \eqref{eq:risk} is equivalent to $\rho_t(\vartheta_t\cdot\Delta S_{t+1})\le x_t-c_t$. 
	In addition, $\rho_t$ is $\mathcal{F}_t$-stable since it is $\mathcal{F}_t$-convex (see \cite[Lemma 4.3]{cheridito2015conditional}).
	Hence we consider the (wealth-dependent) control set
	\[
	\Theta_t(x_t):=\left\{z_t=(\vartheta_t,c_t)\in L^0_t(\mathbb{R}^d\times \mathbb{R}_+): \rho_t(\vartheta_t\cdot\Delta S_{t+1})\le x_t-c_t\mbox{ and }0\le c_t\le x_t \right\},
	\] 
	which is $\mathcal{F}_t$-stable. Suppose that for every $\vartheta\in L^0_t(\mathbb{R}^d)$  one has $\mathbb{P}(\vartheta\cdot\Delta S_{t+1}<0\mid\mathcal{F}_t)>0$ on $\{\vartheta\neq 0\}$ and therefore $\mathbb{P}(\vartheta\cdot\Delta S_{t+1}>0\mid\mathcal{F}_t)>0$ on $\{\vartheta\neq 0\}$. 
	Moreover,
	we assume that $\rho_t(\vartheta\cdot \Delta S_{t+1})\in L^0_t$ for all $\vartheta\in L^0_t(\mathbb{R}^d)$, and
	$\rho_t$ is $\mathcal{F}_t$-\emph{sensitive to large losses}, i.e. $\lim_{m\to\infty}\rho_t(m y)=+\infty$ on $\{\mathbb{P}(y<0\mid\mathcal{F}_t)>0\}$. 
	Then the control set $\Theta_t$ satisfies (i) and (ii) of Proposition \ref{prop:controlfinite}.	
	Indeed, consider the function $f_t\colon  L^0_t(\mathbb{R}^d\times\mathbb{R}_+)\times L^0_t\to L^0_t$ defined as $f_t(\vartheta,c,x):=\rho_t(\vartheta\cdot\Delta S_{t+1})+c - x$, which is $\mathcal{F}_t$-convex and therefore sequentially continuous by \cite[Theorem 7.2]{cheridito2015conditional}. Hence it follows that 
	\[
	\big\{(x,\vartheta,c)\in L^0_t\times L^0_t(\mathbb{R}^d\times\mathbb{R}_+):(\vartheta,c)\in\Theta_t(x)\big\}=\big\{(x,\vartheta,c)\in L^0_t\times L^0_t(\mathbb{R}^d\times\mathbb{R}_+):f_t(\vartheta,c,x)\le 0\big\}
	\]
	is $\mathcal{F}_t$-convex and sequentially closed, which shows (i). 
	As for (ii) let $(x_n)$ be a sequence in $L^0_t$ such that $x_n\to x\in L^0_t$ a.s. For
	$\bar x:=\sup_n x_n\in L^0_t$ one has  
	\[
	\Theta_t(x_n)\subset \Theta_t(\bar x)
	\]
	for all $n\in\mathbb{N}$. Hence, it remains to show that $\Theta_t(\bar x)$ is $\mathcal{F}_t$-bounded, i.e.~there is $M\in L^0_{t,+}$ such that $d_{L^0_t(\mathbb{R}^d)}(\vartheta,0)+c\le M$ for all $(\vartheta,c)\in\Theta_t(\bar x)$. Since $\Theta_t(\bar x)$ contains $(0,0)\in L^0_t(\mathbb{R}^d\times\mathbb{R}_+)$,
	by \cite[Theorem 3.13]{cheridito2015conditional} it is enough to show that for each $(\vartheta,c)\in\Theta_t(\bar x)$ with $(\vartheta,c)\neq(0,0)$ there exists $k\in\mathbb{N}$ such that $k(\vartheta,c)\notin\Theta_t(\bar x)$.
	If $c\neq 0$ this is obvious. Otherwise, one has $\mathbb{P}(\vartheta\neq 0)>0$, in which case $\lim_{m\to\infty}\rho_t(m\vartheta\cdot \Delta S_{t+1})=+\infty$ on $\{\vartheta\neq 0\}$. 
	
	By Proposition \ref{prop:controlfinite} and Theorem \ref{th:existence} it follows that for every recursive utility function with backward generators $(u_t)$, $t=1,\dots,T$,
	satisfying \ref{u1}-\ref{u3} and $x_0>0$, there exists a global optimizer $((x^\ast_s)_{s=1}^T,(\vartheta^\ast_s,c^\ast_s)_{s=0}^{T-1})\in C_0(x_0)$ of the utility maximization problem \eqref{opt:global} satisfying the local criterion \eqref{eq:rec}. 
\end{example}


\subsection{Measurable dimension} 
Suppose that $Z_t$ is the conditional Euclidean space $L^0_t(\mathbb{R})^{d_t}$ with measurable dimension $d_t=d_t(x)\in L^0_t(\mathbb{N})$  that depends on the parameter $x\in X_t$ (see Example \ref{rem:extensions} for the definition of the conditional Euclidean space with measurable dimension).    
Let $d_t\colon X_t\to L^0_t(\mathbb{N})$ be an $\mathcal{F}_t$-stable and sequentially continuous, where $L^0_t(\mathbb{N})$ is endowed with the $\mathcal{F}_t$-conditional metric which extends the discrete metric on $\mathbb{N}$.  
The control set $\Theta_t$ is chosen such that 
\begin{enumerate}
	\item[(c1)] $\emptyset\neq\Theta_t(x)\subset  L^0_t(\mathbb{R})^{d_t(x)}$ for all $x\in X_t$,
	\item[(c2)] $\Theta_t$ is $\mathcal{F}_t$-stable, i.e. 
	\[
	\Theta_t\Big(\sum_k 1_{A_k}x_k\Big)=\sum_k 1_{A_k} \Theta_t(x_k)\subset  L^0_t(\mathbb{R})^{d_t(\sum_k 1_{A_k} x_k)}
	\]
	for all $(A_k)\in \Pi_t$ and every sequence $(x_k)$ in $X_t$,
\end{enumerate}
are satisfied. 
\begin{remark}\label{rem:measurabledimension}
	Since $Z_t=L^0_t(\mathbb{R})^{d_t(x)}$ depends on the state $x\in X_t$ we are in a more general setting as the main Theorem \ref{th:existence}. However, since $L^0_t(\mathbb{N})$ is endowed with the conditional discrete metric, for every sequence $(x_n)$ in $X_t$ such that $x_n\to x\in X_t$ there exists $n_0\in L^0_t(\mathbb{N})$ such that $d_t(x_n)=d_t(x)$ for all $n\ge n_0$.
	In particular, $L^0_t(\mathbb{R})^{d_t(x_n)}=L^0_t(\mathbb{R})^{d_t(x)}$ for all $n\ge n_0$ and Theorem \ref{th:existence} still holds true by exploring the arguments on $z_n\in\Theta_t(x_n)$ for sequences $x_n\to x$ a.s.~in the conditional space $L^0_t(\mathbb{R})^{d_t(x)}$.
\end{remark}

A variant of Proposition \ref{prop:controlfinite} for control sets with measurable dimension can be formulated as follows.
\begin{proposition}\label{prop:controlfinite1}
	Suppose that for each $t=0,\dots,T-1$, the control set $\Theta_t$ satisfies (c1), (c2) and the following conditions:
	\begin{itemize}
		\item[(i)] $\big\{(x,z)\in X_t\times L^0_t(\mathbb{R})^{d_t(x)}:z\in\Theta_t(x)\big\}$ is sequentially closed. 
		\item[(ii)] For every  sequence $(x_n)$ in $X_t$ with $x_n\to x\in X_t$ a.s.~there exists $M\in L^0_{t,+}$ and a measurable subsequence $n_1<n_2<\cdots$ in $L^0_t(\mathbb{N})$ such that $d_{L^0_t(\mathbb{R}^{d_t(x)})}(z,0)\le M$
		for every $z\in\bigcup_{k\ge k_0} \Theta_t(x_{n_k})$ for some $k_0\in L^0_t(\mathbb{N})$ such that $\Theta_t(x_{n_k})\subset  Z_t(d_t(x))$ for all $k\ge k_0$.
	\end{itemize}
	Then, the control set $\Theta_t$ satisfies (c1)-(c4).
\end{proposition}
\begin{proof}
	Let $(x_n)$ in $X_t$ be a sequence such that $x_n\to x\in X_t$ a.s., and $z_n\in \Theta_t(x_n)$. By Remark \ref{rem:measurabledimension} there exists a measurable subsequence $n_1<n_2<\cdots$ with $n_k\in L^0_t(\mathbb{N})$ such that
	$z_{n_k}\in \Theta_t(x_{n_k})\subset  L^0_t(\mathbb{R})^{d_t(x)}$ for all $k$.
	Hence, we can argue similar as in the proof of Proposition \ref{prop:controlfinite}. 
\end{proof}	

\begin{example}\label{ex:measurabledim}
Consider a portfolio maximization problem, where the number of traded assets depends on past decision. More precisely, given a portfolio
$x_t=z_{t-1}=(\vartheta_{t-1},d_{t-1})\in L^0_{t-1}(\mathbb{R})^{d_{t-1}}\times L^0_{t-1}(\mathbb{N})$ chosen at time $t-1$ (with initial value $x_{-1}=(\vartheta_{-1},d_{-1})\in \mathbb{R}^{d_{-1}}\times\mathbb{N}$), the investor can rebalance the portfolio at time $t$ to
\[
x_{t+1}=z_t=(\vartheta_t,d_t)\in \Theta_t(x_t)\subset  L^0_t(\mathbb{R})^{d_{t-1}}\times L^0_t(\mathbb{N}).
\]
Here, the state spaces and the control spaces $X_{t+1}=Z_t=L^0_t(\mathbb{R})^{d_{t-1}}\times L^0_t(\mathbb{N})$ both depend on the past decision
$d_{t-1}$. In line with Remark \ref{rem:measurabledimension}  the convergence $x^n_t=(\vartheta^n_{t-1},d^n_{t-1}) \to x_t=(\vartheta_{t-1},d_{t-1})$ is understood as $\vartheta^n_{t-1}\to \vartheta_{t-1}$ a.s.~in the conditional metric space $L^0_t(\mathbb{R})^{d_{t-1}}$, since $d^n_{t-1}=d_{t-1}$ for all $n\ge n_0$ for some $n_0\in L^0_{t-1}(\mathbb{N})$. Suppose that the control set $\Theta_t$ satisfies (c1), (c2) as well as (i) and (ii) of Proposition \ref{prop:controlfinite1}. Then, along the same argumentation as in  Proposition \ref{prop:controlfinite1} it follows that $\Theta_t$ satisfies (c1)--(c4). Since $v_t(x_t,z_t):=z_t$ satisfies (v1) and (v2), Theorem \ref{th:existence} is applicable whenever the backward generators $u_t$ satisfy (u1)-(u3).

The measurable dimension depending on past decisions allows for instance to add new assets at time $t$ (i.e.~$d_t>d_{t-1}$) which are traded at $t+1$. Notice that $\Theta_t(\vartheta_{t-1},d_{t-1})$ denotes the set of all attainable portfolios at time $t$. For instance, let $S_t\in L^0_{t,++}(\mathbb{R}^d)$
be a price process with fixed $d\in\mathbb{N}$.
Without frictions and short-selling constraints
one has $\Theta_t(\vartheta_{t-1}):=\{\vartheta_t\in L^0_{t,+}(\mathbb{R}^d):\vartheta_t\cdot S_t=\vartheta_{t-1}\cdot S_t\}$
which satisfies (c1)-(c4). Transaction costs can be included into the model by considering 
$\Theta_t(\vartheta_{t-1}):=\{\vartheta_t\in L^0_{t,+}(\mathbb{R}^d):\vartheta_t-\vartheta_{t-1}\in C_t\}$ for a solvency region $C_t\subset  L^0_t(\mathbb{R}^d)$, see e.g.~\cite{pennanen2010hedging} for a discussion of different market models.
Notice that the solvency regions can be modeled state-dependently with measurable dimension $d_t\in L^0_t(\mathbb{N})$.
\end{example}


\section{Unbounded control sets}\label{sec:unbounded}
In this section we consider unbounded control sets $\Theta_t\equiv L^0_t(\mathbb{R}^d)$ and do not assume constraints on the controls, but derive (c3) and (c4) for upper-level sets of $y_t$ as a result of stronger assumptions on the forward and backward generators. 
In particular, we additionally need that the backward generators are $\mathcal{F}_t$-sensitive to large losses and increasing in the first argument, see (u5) and (u2') below.
Suppose that the forward generators
\[
v_t:L^0_t \times L^0_t(\mathbb{R}^d) \rightarrow L^0_{t+1}, \quad t=0,1,\ldots,T-1, 
\]
satisfy \ref{v1}, \ref{v2} and 
\begin{enumerate}
	\item[(v3)] $v_t$ is increasing in the first component,
	\item[(v4)] $v_t(x,\lambda z + (1-\lambda)z^\prime)\geq \lambda v_t(x,z) + (1-\lambda) v_t(x,z^\prime)$
	for all $x\in L^0_t$, $z,z^\prime \in L^0_t(\mathbb{R}^d)$ and $\lambda \in L^0_t$ with $0\leq\lambda\leq 1$,
	\item[(v5)] $\mathbb{P}(v_t(x,z)<x\mid\mathcal{F}_t)>0$
	on $\{z\neq 0\}$ for all  $x\in L^0_t$ and $z\in L^0_t(\mathbb{R}^d)$, 
	\item[(v6)] $v_t(x,0)=x$ for all $x\in L^0_t$. 
\end{enumerate}
As for the backward generators, let $u_T\colon L^0_T\to L^0_T$ be the identity mapping, and 
\[
u_t:L^0_t\times \underbar{L}^0_{t+1}\times L^0_t(\mathbb{R}^d)\rightarrow\underbar{L}^0_t, \quad t=0,\ldots, T-1, 
\]
satisfy \ref{u1} and \ref{u3} as well as
\begin{enumerate}
	\item[(u2')] increasing in the first and second component, 
	\item[(u4)] $u_t(x,\lambda y + (1-\lambda)y^\prime,\lambda z + (1-\lambda)z^\prime)\geq \min\left\{u_t(x,y,z),u_t(x,y^\prime,z^\prime)\right\}$
	for all $x\in L^0_t(I)$, $y,y^\prime\in \underbar{L}^0_{t+1}$, $z,z^\prime\in L^0_t(\mathbb{R}^d)$ and $\lambda\in L^0_t$ with $0\leq\lambda\leq 1$,
	\item[(u5)] $u_t(x,y+c,z)=u_t(x,y,z)+c$  for all $x\in L^0_t$, $y\in\underbar{L}^0_{t+1}$, $z\in L^0_t(\mathbb{R}^d)$ and $c\in L^0_t$,
	\item[(u6)] $\lim_{m\to \infty}u_t(x,m y,m z)=-\infty$ a.s.~on $\{\mathbb{P}(y<0\mid\mathcal{F}_t)>0\}$
	for every $z\in L^0_t(\mathbb{R}^d)$ and $y\in\underbar{L}^0_{t+1}$, 
	\item[(u7)]  $u_t(x,0,0)=0$ for all $x\in L^0_t$. 
\end{enumerate}
Let $y_t\colon L^0_t\to \underbar{L}^0_t$ be given as in 
\eqref{opt:global} where
\[
C_t(x_t):=\left\{\big((x_s)_{s=t+1}^T,(z_s)_{s=t}^{T-1}\big)
\colon x_{s+1}=v_s(x_s,z_s),z_s\in L^0_t(\mathbb{R}^d)\textnormal{ for all }s=t,\ldots,T-1\right\}.
\]
Then the following variant of Theorem \ref{th:existence} holds.
\begin{proposition}\label{sl}
	Suppose that  \ref{v1}--(v6) and \ref{u1},(u2'),(u3)--(u7) are fulfilled, and 
	there exists a constant $K>0$ such that 
	\[
	\sup_{z\in L^0_t(\mathbb{R}^d)} u_t(x,v_t(x,z),z) - x \leq K
	\]
	for all $t=0,\ldots,T-1$ and $x\in L^0_t$. 
	Then the functions $y_t\colon L^0_{t}\rightarrow\underbar{L}^0_t$ are $\mathcal{F}_t$-stable, increasing and sequentially upper semi-continuous for all $t=0,\dots, T$, and can be computed by backward recursion	
	\begin{align*}
	y_T(x_T)&=u_T(x_T)=x_T\\
	y_t(x_t)&=\max_{z_t\in L^0_t(\mathbb{R}^d)} u_t(x_t,y_{t+1}(v_t(x_t,z_t)),z_t), \quad t=0,\ldots,T-1.
	\end{align*}
	Moreover, for every $x_t\in L^0_t$ the process
	$((x^\ast_s)_{s=t}^T,(z^\ast_s)_{s=t}^{T-1})$ given by
	$x_t^\ast=x_t$ and forward recursion
	\begin{equation}\label{eq:rec1}
	x^\ast_{s+1}=v_s(x^\ast_s,z^\ast_s)\quad\mbox{where}\quad z_s^\ast\in \mathop{\rm argmax}_{z_s\in L^0_t(\mathbb{R}^d)} u_s\big(x_s^\ast,y_{s+1}(v_t(x^\ast_s,z_s)),z_s\big),\quad s=t,\dots T-1,
	\end{equation} 
	satisfies $((x^\ast_s)_{s=t+1}^T,(z^\ast_s)_{s=t}^{T-1})\in C_t(x_t)$ and 	
	\[
	y_t(x_t)= u_t(x_t,\cdot,z^\ast_t)\circ\cdots\circ u_{T-1}(x^\ast_{T-1},\cdot,z^\ast_{T-1})\circ u_T (x^\ast_T). 
	\]	
\end{proposition}  
\begin{proof}
	The proof is similar to Theorem \ref{th:existence}. However, since the control set is not compact we have to argue differently to show the existence of \eqref{eq:ft}, i.e.~that the supremum in
	\[
	y_t(x_t):=\sup_{z\in L^0_t(\mathbb{R}^d)} u_t\big(x_t,y_{t+1}(v_t(x_t,z)),z\big),\quad x_t\in L^0_t
	\]
	is attained. To do so, we first show that
	\begin{equation}\label{bound1}
	0\le y_t(x) - x \leq K_t\quad \mbox{for all }x\in L^0_t, 
	\end{equation}
	where $K_t:=(T-t)K$ for all $t=0,1,\ldots,T$. For $t=T$, one has $y_T(x)-x=0$. By induction suppose that $y_{t+1}(x)-x\leq (T-t)K_{t+1}$. 
	Then, by (u2') and (u5) for every $z\in L^0_t(\mathbb{R}^d)$  one has
	\begin{align*}
	&u_t\big(x,y_{t+1}(v_t(x,z)\big),z)-x=u_t\big(x,y_{t+1}(v_t(x,z))-v_t(x,z)+v_t(x,z),z\big)-x \\
	&\leq u_t(x,v_t(x,z),z) - x + K_{t+1} \leq K+K_{t+1}=K_t
	\end{align*}
	so that $y_t(x)-x\le K_t$. As for the lower bound, suppose by induction that $x\le y_{t+1}(x)$. By (v6), (u2'), (u5) and (u7) it follows that
	\begin{align*}
	y_t(x)\geq u_t\big(x,y_{t+1}(v_t(x,0)),0\big)\geq u_t\big(x,y_{t+1}(x),0\big) \geq u_t\big(x,x,0\big) = u_t\big(x,0,0\big)+x= x.
	\end{align*}
	Fix $x\in L^0_t$. For every $z\in L^0_t(\mathbb{R}^d)$ which satisfies 
	$u_t(x,y_{t+1}(v_t(x,0)),0)\le u_t(x,y_{t+1}(v_t(x,z)),z)$,
	it follows from \eqref{bound1} (u5), (u7) and (v6) that
	\begin{align*}
	x&=u_t(x,x,0)\le u_t(x,y_{t+1}(x),0) \le u_t\big(x,y_{t+1}(v_t(x,0)),0\big)\\
	&\le u_t\big(x,y_{t+1}(v_t(x,z)),z\big)
	\le u_t\big(x,v_t(x,z),z\big)+K_{t+1}.
	\end{align*}
	This shows that
	\[
	y_t(x)=\sup_{z\in\Theta_t(x)} u_t\big(x,y_{t+1}(v_t(x,z)),z\big)
	\]
	for the $\mathcal{F}_t$-stable set
	\[
	\Theta_t(x):=\left\{z\in L^0_t(\mathbb{R}^d)\colon u_t(x,v_t(x,z),z) \geq x - K_{t+1}\right\}. 
	\]
	It remains to show that $\Theta_t$ satisfies (c1)-(c4). To that end, we verify (i) and (ii) of Proposition \ref{prop:controlfinite}. By (u3) and (v2) it follows that the set
	\[
	\big\{(x,z)\in L^0_t\times L^0_t(\mathbb{R}^d):z\in\Theta_t(x)\big\}
	\]
	is sequentially closed, which shows (i) of Proposition \ref{prop:controlfinite}. As for  (ii) of Proposition \ref{prop:controlfinite} let $(x_n)$ be a sequence in $L^0_t$ such that $x_n\to x\in L^0_t$ a.s. Defining $\underline x:=\inf_n x_n\in L^0_t$ as well as $\bar x:=\sup_n x_n\in L^0_t$, it follows from (u2') and (v3) that 
	\[
	\Theta_t(x_n)\subset \left\{z\in L^0_t(\mathbb{R}^d)\colon u_t(\bar x,v_t(\bar x,z),z) \geq \underline x - K_{t+1}\right\}=:\Theta_t(\underline x,\bar x)
	\]
	for all $n\in\mathbb{N}$. Moreover, by (u4) and (v4) the set $\Theta_t(\underline x,\bar x)$ is $\mathcal{F}_t$-convex. It remains to show that 
	there exists $M\in L^0_t$ such that
	$d_{L^0_t(\mathbb{R}^d)}(z,0)\le M$
	for all $z\in \Theta_t(\underline x,\bar x)$. This $L^0_t$-boundedness of $\Theta_t(\underline x,\bar x)$ would follow from \cite[Theorem 3.13]{cheridito2015conditional}, if for all $z\in \Theta_t(\underline x,\bar x)$ with $z\neq 0$, there exists $A\in\mathcal{F}_t$ with $\mathbb{P}(A)>0$ such that   
	\begin{equation}\label{eq:sl}
	\lim_{m\to \infty} u_t(\bar x, v_t(\bar x,m z), m z)=-\infty\mbox{ a.s.} \quad\textnormal{ on }A. 
	\end{equation}
	Indeed, since by (v5) one has $\mathbb{P}(v_t(\bar x,z)<\bar x\mid\mathcal{F}_t)>0$ on $\{z\neq 0\}$,
	there exists $l\in\mathbb{N}$ such that
	\[
	A:=\big\{\mathbb{P}\big(|\bar x| + l (v_t(\bar x,z)- \bar x)<0 \mid\mathcal{F}_t\big)>0\big\}\in\mathcal{F}_t
	\]
	satisfies $\mathbb{P}(A)>0$. By (v4) it follows that
	\[
	v_t(\bar x,z)\ge\frac{1}{m}v_t(\bar x,mz)+\frac{m-1}{m}v_t(\bar x,0)
	\]
	which by (v6) implies $m\big(v_t(\bar x,z)-\bar x\big)\ge v_t(\bar x,mz)-\bar x$
	for all $m\in\mathbb{N}$.  This shows that
	\begin{align*}
	u_t(\bar x,v_t(\bar x,m z),m z) &\leq u_t(\bar x,|\bar x| + v_t(\bar x,m z)- \bar x,m z) \leq u_t\left(\bar x,\frac{m}{l}\left(|\bar x| + l(v_t(\bar x,z)- \bar x)\right),m z\right)
	\end{align*}
	for all $m\in\mathbb{N}$ large enough. Hence, the condition (u6) implies \eqref{eq:sl}.

\end{proof}

\begin{example}\label{ex:expectedutility}
	Let  $(S_t)_{t=0}^T$ be a $\mathbb{R}^d$-valued adapted stochastic process modeling the discounted stock prices of a financial market model.
	Given a trading strategy  $\vartheta_t\in L^0_t(\mathbb{R}^d)$, $t=0,\dots,T-1$, and an initial investment $x_0\in L^0_0$ we define recursively the wealth process
	\[
	x_{t+1}=v_t(x_t,\vartheta_t):=x_t + \vartheta_t\cdot \Delta S_{t+1}, 
	\quad t=0,\ldots,T-1, 
	\]
	where $\Delta S_{t+1}:=S_{t+1} - S_{t}$ denotes the stock price increment. 
	We assume the following no-arbitrage condition
	(which includes a relevance condition on the market model) \[\vartheta\cdot\Delta S_{t+1}\ge 0 \mbox{ a.s.~for }\vartheta\in L^0_t(\mathbb{R}^d)\quad\mbox{implies}\quad \vartheta= 0\mbox{ a.s.}\]
	for all $t=0,\dots,T-1$.
	Then the forward generator $v_t:L^0_t\times L^0_t(\mathbb{R}^d)\to L^0_{t+1}$ satisfies (v1)--(v6). As for the backward generators, let $u_T\colon L^0_T\to L^0_T$ be the identity and 
	\[
	u_t:L_t^0\times \underbar{L}^0_{t+1}\rightarrow \underbar{L}^0_t, \quad u_t(x,y):=\frac{1}{\gamma_t(x)}g_t(\gamma_t(x)y), \quad t=0,\ldots,T-1, 
	\]
	where $g_t \colon \underbar{L}^0_{t+1}\rightarrow \underbar{L}^0_t$ is increasing, $\mathcal{F}_t$-concave, $\mathcal{F}_t$-translation invariant, sequentially upper semi-continuous, $g_t(0)=0$ and $\lim_r g_t(ry)=-\infty$ on $\{\mathbb{P}(y<0\mid\mathcal{F}_t)>0\}$. The function $\gamma_t\colon L^0_t\to L^0_{t,++}$ is $\mathcal{F}_t$-stable, decreasing and sequentially continuous and models the risk aversion 
	depending on the wealth $x_t$ at time $t$. Then, $u_t$ satisfies the conditions \ref{u1},(u2'),(u3)--(u7). 
	We only verify (u2') and (u3). 
	To prove (u2') take $x_1\leq x_2$ and $y_1\leq y_2$. Let $\beta_i:=1/\gamma_t(x_i)$, for $i=1,2$. By using the monotonicity and $\mathcal{F}_t$-concavity of $g_t$, we have 
	\[
	g_t\Big(\frac{y_2}{\beta_2}\Big)\geq g_t\Big(\frac{y_1}{\beta_2}\Big)\geq \frac{\beta_1}{\beta_2}g_t(\frac{y_1}{\beta_1}) +  \frac{\beta_2-\beta_1}{\beta_2}g_t(0)
	= \frac{\beta_1}{\beta_2} g_t\Big(\frac{y_1}{\beta_1}\Big). 
	\]
	Multiplying by $\beta_2$ we obtain  $u_t(x_2,y_2)\geq u_t(x_1,y_1)$.  
	Due to the monotonicity of $u_t$ it suffices to verify (u3) for decreasing sequences. 
	Indeed, suppose that ${x}_k\searrow x$ a.s. and ${y}_k\searrow y$ a.s. 
	Then, by the monotonicity of $u_t$ we have
	\[
	g_t(\gamma_t({x}_k){y}_k)\geq\frac{\gamma_t({x}_k)}{\gamma_t(x)}g_t(\gamma_t(x)y)\quad\text{ for all }k.
	\] 
	Thus, by using that $g_t$ is sequentially upper semi-continuous and $\gamma$ is sequentially continuous we obtain
	\[
	g_t(\gamma_t(x)y)\geq \underset{k\to\infty}\limsup g_t(\gamma_t({x}_k){y}_k)\geq \underset{k\to\infty}\liminf g_t(\gamma_t({x}_k){y}_k)
	\geq \underset{k\to\infty}\lim \frac{\gamma_t({x}_k)}{\gamma_t(x)}g_t(\gamma_t(x)y)=g_t(\gamma_t(x)y).  
	\] 
	This shows that $\underset{k\to\infty}\lim g_t(\gamma_t({x}_k){y}_k)=g_t(\gamma_t(x)y)$, and therefore $\underset{k\to\infty}\lim u_t(x_k,y_k)=u_t(x,y)$.

	Given the wealth process $(x_t)_{t=0}^T$, define the backward process 
	\[
	y_t(x_t)=\underset{((x_s)_{s=t}^T,(\vartheta_s)_{s=t}^{T-1})\in C_t(x_t)}\sup u_t(x_t,\cdot)\circ\ldots\circ u_{T-1}(x_{T-1},\cdot)\circ u_T (x_T),\quad t=0,\ldots,T-1,
	\]
	where $C_t(x_t):=\left\{((x_s)_{s=t}^T,(\vartheta_s)_{s=t}^{T-1}) \colon x_{s+1}=x_s+\vartheta_{s+1}\cdot \Delta S_{s+1}\textnormal{, for }s=t,\ldots,T-1\right\}$.
	By induction, one can verify that $y_t(x+c)=y_t(x)+c$ for every $c\in L^0_{t-1}$ with $t=1,\ldots,T$. 
	Suppose there exists $K>0$ such that 
	\begin{equation}\label{cond:ex}
	u_t\big(x,v_t(x,\vartheta),\vartheta\big)-x\le \frac{1}{\gamma_t(x)}g_t\big(\gamma_t(x)\vartheta\Delta S_{t+1}\big)\le K
	\end{equation}
	for all $t=0,\dots,T-1$, $\vartheta\in L^0_t(\mathbb{R}^d)$, and $x\in L^0_t$.
	Then, it follows from Proposition \ref{sl} that 
	\[
	y_0(x_0)=\underset{((x_t)_{t=0}^T,(\vartheta_t)_{t=0}^{T-1})\in C_0(x_0)} \sup u_0(x_0,\cdot)\circ\ldots\circ u_{T-1}(x_{T-1},\cdot)\circ u_T (x_T)
	\]
	is attained for every $x_0\in L^0_0$.
	For instance one could think of the dynamic entropic preference functional with generators \[-\frac{1}{\gamma_t}\log\big(\mathbb{E}[\exp(-\gamma_t y)\mid\mathcal{F}_t]\big) \] 
	where the local risk aversion coefficient $\gamma_t=\gamma_t(x_t)$ depends on the current wealth $x_t$. Notice that $\lim_{m\to\infty}-\log(\mathbb{E}[\exp(-my)\mid\mathcal{F}_t])=-\infty$ on $\{\mathbb{E}[y<0\mid\mathcal{F}_t]>0\}$. 
	
	
\end{example}

We conclude this section with a \emph{wealth-dependent dynamic risk sharing problem}. Let $\mathbb{A}$ be a finite set of agents. 
Each agent $a\in\mathbb{A}$ is endowed with a wealth process $(H^a_t)_{t=0}^T$ with $H^a_t\in L^0_{t,++}$. 
The aim is to share optimally the aggregated endowment process $H_t=\sum_{a\in \mathbb{A}} H^a_t$, $t=0,\ldots,T$, with respect to a dynamic wealth-dependent utility. 
The utilities under consideration are of the form 
\[
u_t \colon L^0_{t,++}\times L^0_{t+1,++}\to \underbar{L}^0_t, \quad u_t(x,y):=x g_t(y/x)
\]
where $g_t:\underbar{L}^0_{t+1}\rightarrow \underbar{L}^0_t$ is an $\mathcal{F}_t$-concave, increasing and  sequentially upper semi-continuous generator with $g_t(0)=0$ for all $t=0,\ldots,T-1$, and $u_T\colon L^0_{T,++}\to \underbar{L}^0_T$ is the identity. 

\begin{proposition}\label{risksharing}
	The wealth-dependent dynamic optimal risk sharing problem 
	\begin{align*}
	y_t((&H^a_t)_{a\in\mathbb{A}}) \\
	& =\sup\left\{ \sum_{a\in\mathbb{A}} u_t(H^a_t, \cdot) \circ \ldots \circ u_{T-1}(x^a_{T-1}, x^a_T) \colon \sum_{a\in\mathbb{A}} x^a_s=H_s, \, x^a_s\in L^0_{s,++},\,s=t+1,\ldots,T \right\} 
	\end{align*}
	has the optimal solution 
	\[
	x^{a,\ast}_s=\frac{H^a_t}{H_t}H_s, \quad a\in\mathbb{A}, \, s=t+1,\ldots,T. 
	\]
	Moreover, the function $y_t\colon L^0_{t}\left((0,\infty)^{|\mathbb{A}|}\right)\to \underbar{L}^0_t$ is $\mathcal{F}_t$-stable, increasing and sequentially upper semi-continuous for $t=0,\ldots,T$, where $|\mathbb{A}|$ denotes the cardinality of $\mathbb{A}$.
\end{proposition}
\begin{proof}
	Define  $\bar y_T \colon L^0_{T}\left((0,\infty)^{|\mathbb{A}|}\right)\to L^0_T$, $\bar y_T((x^a)_{a\in\mathbb{A}})=\sum_{a\in\mathbb{A}}x^a$, and for $t=0,\ldots, T-1$, let
	\[
	\bar y_t\colon L^0_{t}\left((0,\infty)^{|\mathbb{A}|}\right)\to \underbar{L}^0_t, \quad  \bar y_t\left(\left(x^a\right)_{a\in\mathbb{A}}\right)=\left(\sum_{a\in\mathbb{A}} x^a \right)g_t\left(\frac{\bar y_{t+1}\left((H^a_{t+1})_{a\in\mathbb{A}}\right)}{\sum_{a\in\mathbb{A}} x^a}\right).
	\]
	By backward induction, it can be checked  that $\bar y_t$ is $\mathcal{F}_t$-stable, increasing and sequentially upper semi-continuous. We show that 
	\begin{equation}\label{eq: I}
	\sum_{a\in\mathbb{A}} u_t(H^a_t,\cdot)\circ u_{t+1}(x^{a,\ast}_{t+1},\cdot)\circ\ldots\circ u_{T-1}(x^{a,\ast}_{T-1}, x^{a,\ast}_T)=\bar y_t((H^a_t)_{a\in\mathbb{A}})
	\end{equation}
	and 
	\begin{equation}\label{eq: II}
	\sum_{a\in\mathbb{A}} u_t(H^a_t,\cdot)\circ u_{t+1}(x^{a}_{t+1},\cdot)\circ\ldots\circ u_{T-1}(x^{a}_{T-1}, x^{a}_T) \leq \bar y_t((H^a_t)_{a\in\mathbb{A}})
	\end{equation}
	for all $(x^a_s)_{a\in\mathbb{A}}$ such that $\sum_{a\in\mathbb{A}} x^a_s=H_s$,  $s=t+1,\ldots,T$. 
	It would follow from \eqref{eq: I} and \eqref{eq: II} that 
	\[
	x^{a,\ast}_s=\frac{H^a_t}{H_t}H_s, \quad a\in\mathbb{A}, \, s=t+1,\ldots,T 
	\]
	is an optimal solution and that $y_t((H^a_t)_{a\in\mathbb{A}})=\bar y_t((H^a_t)_{a\in\mathbb{A}})$. 
	By induction, it can be checked that 
	\begin{equation}\label{eq: III}
	u_s(x^{a,\ast}_s,\cdot)\circ u_{s+1}(x^{a,\ast}_{s+1},\cdot)\circ\ldots\circ u_{T-1}(x^{a,\ast}_{T-1}, x^{a,\ast}_T)=x^{a,\ast}_s \frac{\bar y_s((H^a_s)_{a\in\mathbb{A}})}{H_s} 
	\end{equation}
	for all $a\in\mathbb{A}$ and $s=t,\ldots,T$ where we put $x^{a,\ast}_t=H^a_t$. 
	By summing up \eqref{eq: III} at $s=t$ over $\mathbb{A}$ we obtain \eqref{eq: I}. 
	We prove \eqref{eq: II} by backward induction. 
	It is true at $T$ by definition. 
	Let $t\leq s< T$. 
	Then 
	\begin{align*}
	&\sum_{a\in\mathbb{A}} u_s(x^{a}_s,\cdot)\circ u_{s+1}(x^a_{s+1},\cdot)\circ\ldots\circ u_{T-1}(x^a_{T-1},x^{a}_T)\\
	&=H_s \sum_{a\in\mathbb{A}} \frac{x^{a}_s}{H_s} g_s\left(\frac{1}{x^{a}_s} u_{s+1}(x^{a}_{s+1},\cdot)\circ\ldots\circ u_{T-1}(x^{a}_{T-1},x^a_T)\right)\\
	&\leq H_s g_s\left(\sum_{a\in\mathbb{A}} \frac{x^{a}_s}{H_s} \frac{1}{x^{a}_s} u_{s+1}(x^{a}_{s+1},\cdot)\circ\ldots\circ u_{T-1}(x^{a}_{T-1},x^a_T)\right)\\
	&\leq H_s g_s\left(\frac{y_{s+1}((H^a_{s+1})_{a\in\mathbb{A}})}{H_s}\right)=y_s((H^a_s)_{a\in\mathbb{A}})
	\end{align*}
	where the first inequality follows from $\mathcal{F}_s$-concavity and the last one by monotonicity of $g_s$ and the induction hypothesis. 
\end{proof}

\section{Connection to random set theory}\label{sec:randomset}

The aim of this section is to discuss methodological similarities and differences of conditional analysis on the one hand and measurable selections and random set theory on the other hand.  
Both techniques were developed to deal with measurability.   
The two approaches can be briefly described as follows. 

Measurable selections and random set theory are established on the basis of classical analysis.  
Therefore, they seek a set-valued formalization to which classical theorems can be applied pointwisely.   
The r\^ole of measurable selection theorems is then to secure measurability under pointwise application of classical theorems. 
This is achieved under topological assumptions.   

Conditional analysis relies on a measurable version of classical results which can be directly applied to sets of measurable functions.  
Therefore, the formalization mainly consists in describing those sets for which a measurable version of classical results can be proved \cite{drapeau2016algebra}.  
Measurability is then systematically preserved by the application of a conditional version of classical results, that is by construction.   
A measurable version of classical theorems (more generally, a transfer principle \cite{carl2018transfer}) exists under measure-theoretic  assumptions, while the topological restrictions of measurable selection techniques and random set theory can be relaxed. 

In the following, we show that conditional analysis extends measurable selections and random set theory. 
More precisely, we establish a correspondence between basic objects in random set theory and their analogues in conditional analysis under the hypothesis of separability.  
Unless mentioned otherwise, we fix a measurable space $(\Omega,\mathcal{F})$ and a Polish space $E$.  
Recall that a closed-valued map $S\colon \Omega\rightrightarrows E$ (i.e.~$S(\omega)\subset E$ is a closed set for all $\omega\in \Omega$) is  Effros measurable\footnote{
There are other measurability concepts besides Effros measurability, see e.g.~\cite[Section 1.2]{molchanov2005theory}. 
One of them is graph-measurability (i.e.~$\{(\omega,x)\in \Omega\times E\colon x\in S(\omega)\}$ is product-measurable) which is important in applications.   
For closed-valued mappings, graph measurability is equivalent to Effros measurability whenever the underlying measurable space is complete, cf.~e.g.~\cite[Theorem 2.3]{molchanov2005theory}.} whenever $S^{-1}(O):=\{\omega\in \Omega\colon S(\omega)\cap O\neq \emptyset\}\in\mathcal{F}$ for all open sets $O$ in $E$.  
We always assume $S^{-1}(E)=\Omega$.    
A measurable selection of $S$ is a Borel function $x\colon \Omega\to E$ such that $x(\omega)\in S(\omega)$ for all $\omega\in \Omega$.   
The following measurable selection theorem is due to Castaing \cite{castaing67}, where cl denotes closure.     
\begin{theorem}\label{castaing}
	A closed-valued map $S\colon \Omega\rightrightarrows E$ is Effros measurable if and only if there exists a countable family of Borel functions $x_n\colon \Omega\to E$ such that $S(\omega)=\text{cl}\{x_n(\omega)\colon n\in\mathbb{N}\}$ for each $\omega\in\Omega$. 
\end{theorem}
As an auxiliary result, a measurable version of the axiom of choice is needed which is adopted to the present setting in what follows, see \cite{drapeau2016algebra} for the general statement.   
Let $I\subset  L^0(\mathbb{N})$ be a stable set. 
A family $(x_i)_{i\in I}$ of elements of $L^0(E)$ is said to be a \emph{stable family} if $x_{\sum_k 1_{A_k} i_k}=\sum_k 1_{A_k} x_{i_k}$ for all $(A_k)\in\Pi_\mathcal{F}$ and sequences $(i_k)$ in $I$.  
A family $(H_i)_{i\in I}$ of subsets of $L^0(E)$ is said to be a \emph{stable family of stable sets} if each $H_i$ is a stable set and 
\[
H_{\sum_k 1_{A_k} i_k}=\sum_k 1_{A_k} H_{i_k}:=\Big\{\sum_k 1_{A_k} z_k\colon z_k\in H_{i_k}\Big\}
\]
for all $(A_k)\in\Pi_\mathcal{F}$ and sequences $(i_k)$ in $I$. 
\begin{lemma}\label{lem:choice}
	Let $(H_i)_{i\in I}$ be a stable family of stable sets in $L^0(E)$. 
	Then there exists a stable family $(x_i)_{i\in I}$ such that $x_i \in H_i$ for all $i\in I$. 
\end{lemma}
\begin{proof}
	Let 
	\[
	\mathscr{H}=\{(x_j)_{j\in J}\colon x_j\in H_j \text{ for all } j\in J, \, J\subset  I \text{ stable}\}. 
	\]
	As each $H_i\neq \emptyset$ any single-element family $\{x_j\}$ with $x_j\in H_j$ for some $\{j\}\subset  I$ is in $\mathscr{H}$, and thus $\mathscr{H}\neq \emptyset$. 
	Order $\mathscr{H}$ by the relation 
	\[
	(x_j)_{j\in J}\leq (\bar{x}_j)_{j\in \bar{J}} \text{ if and only if } J\subset  \bar{J} \text{ and } x_j=\bar{x}_j \text{ for all } j\in J. 
	\]  
	It is straightforward to check that $(\mathscr{H},\leq)$ is a partially ordered set. 
	Let $(x^\alpha_j)_{j\in J^\alpha}$ be a chain in $\mathscr{H}$. 
	Define 
	\[
	J:=\Big\{\sum_k 1_{A_k} j_k\colon (A_k)\in \Pi_\mathcal{F}, \, j_k\in \cup_\alpha J^\alpha \text{ for each } k\Big\}. 
	\]
	By definition $J\subset  I$ is stable. 
	For $j\in J$, put $x_j=\sum_k 1_{A_k} x_{j_k}^\alpha$. 
	Since $(H_i)$ is a stable family, $(x_j)_{j\in J}\in\mathscr{H}$.  
	By construction $(x^\alpha_j)_{j\in J^\alpha}\leq (x_j)_{j\in J}$ for all $\alpha$. 
	By Zorn's lemma, there is a maximal element $(x^\ast_j)_{j\in J^\ast}\in \mathscr{H}$. 
	By contradiction, suppose there is $i\in I$ such that $\sup_{j\in J^\ast} \{i=j\}\neq\Omega$. 
	Let $\hat{J}=\{1_A i + 1_{A^c} j\colon j\in J^\ast, \, A\in \mathcal{F}\}$, pick some $x_i\in H_i$, and define $\hat{x}_{\hat{j}}=1_{A}x_i +1_{A^c}x^\ast_j$ for $\hat{j}\in \hat{J}$.  
	Then $(\hat{x}_j)_{j\in \hat{J}}$ is an element of $\mathscr{H}$, but $(x^\ast_j)_{j\in J^\ast}< (\hat{x}_j)_{j\in \hat{J}}$.  
\end{proof}
Fix a probability measure $\mathbb{P}$ on $(\Omega,\mathcal{F})$ and complete $\mathcal{F}$ relative to $\mathbb{P}$. 
We identify two closed-valued Effros measurable maps $S_1$ and $S_2$ whenever $S_1(\omega)=S_2(\omega)$ a.s.  
Let $X_S$ denote the set of measurable selections of a set-valued mapping $S$. 
In the following proposition, we construct a set-valued mapping which is associated to a set $X$ in $L^0(E)$, and  which will be denoted by $S_X$.  
\begin{theorem}\label{t:selec}
	Let $S\colon\Omega\rightrightarrows E$ be a closed-valued and Effros measurable mapping, and let $X\subset L^0(E)$ be a stable and sequentially closed set.  
	Then there exist closed-valued and Effros measurable mappings $S_X\colon\Omega\rightrightarrows E$ and $S_{X_S}\colon\Omega\rightrightarrows E$ satisfying the reciprocality relations $S=S_{X_S}$ and $X=X_{S_X}$ respectively.  
\end{theorem}
\begin{proof}
	First, $S_X$ is constructed.
	Second, the reciprocality relations $X=X_{S_X}$ and $S=S_{X_S}$ are established. 
	\begin{itemize}[fullwidth]
		\item[(i)] 
		Let $F=\{q_1,q_2,\ldots\}$ be a countable dense set in $E$.  
		For each $n\in L^0(\mathbb{N})$ and $q\in L^0(F)$, define the random ball  
		\[
		B_{1/n}(q):=\{x\in L^0(E)\colon d(x,q)<1/n \text{ a.s.}\}. 
		\] 
		Put 
		\[
		I=\{ (n,q)\in L^0(\mathbb{N})\times L^0(F)\colon X\cap B_{1/n}(q)\neq \emptyset \}. 
		\]
		Inspection shows that $I$ is a stable subset of $L^0(\mathbb{N})\times L^0(F)$ which can be identified with a stable set in $L^0(\mathbb{N})$ since $F$ is countable.  
		By Lemma \ref{lem:choice}, there is a stable family $(x_i)$ in $L^0(E)$ such that $x_i\in X\cap B_{1/n}(q)$ for each $i=(n,q)\in I$.  
		Next, we construct the largest measurable set $A\in\mathcal{F}$ restricted to which $I$ is conditionally finite, where for $m\in L^0(\mathbb{N})$, we denote by $\{1\leq k\leq m\}:=\{k\in L^0(\mathbb{N})\colon 1\leq k \leq m\}$ a random interval of integers which encodes conditional finiteness.    
		For a set $N$ in $L^0(\mathbb{N})$, denote by $1_A N:=\{1_A n \colon n\in N\}$. 
		Let 
		\[
		\mathcal{E}=\{A\in\mathcal{F}\colon \text{ there are } m\in L^0(\mathbb{N}) \text{ and a stable bijection } f\colon 1_A\{1\leq k\leq m\}\to 1_A I\}. 
		\]
		We want to show that $A_\ast:=\cup \mathcal{E}\in \mathcal{E}$. 
		By \cite[Lemma 1, Chapter 30]{halmos09}, there exists a sequence $(A_n)$ in $\mathcal{E}$ such that $A_\ast=\cup_n A_n$ .  
		Form $B_n=A_n\cap (\cup_{k\leq n} A_k^c)$, each $n$. 
		Then $(B_n)$ is a sequence of elements in $\mathcal{E}$. 
		Indeed, if $f\colon 1_A \{1\leq k\leq m\}\to 1_A I$ is a stable bijection and $B\subset  A$, then $g(1_B k):=1_B f(1_A k)$ defines a stable bijection $g\colon 1_B\{1\leq k\leq m\}\to 1_B I$. 
		Let $f_n\colon 1_{A_n}\{1\leq k\leq m_n\}\to 1_{A_n} I$ be a stable bijection. 
		Then $f_\ast \colon 1_{A_\ast}\{1\leq k\leq \sum_n 1_{B_n} m_n\}\to 1_{A_\ast} I$ defined by $f_\ast(1_{A_\ast} k):=\sum_n 1_{B_n} f_n(1_{A_n} k)$ 		is a stable bijection. 
		Thus $A_\ast\in \mathcal{E}$. 
		By maximality of $A_\ast$, there exists a stable bijection $g_\ast\colon 1_{A_\ast^c} L^0(\mathbb{N})\to 1_{A_\ast^c} I$ since on $A_\ast$ the conditional index $I$ is nowhere conditionally finite. 
		Notice that $\sum_n 1_{B_n} m_n$ can be rearranged as $\sum_k 1_{C_k} l_k$ where $(l_k)$ is a sequence of natural numbers and $(C_k)$ is pairwise disjoint. 
		Now define  
		\[
		S_X(\omega):=
		\begin{cases}
		\text{cl}\left\{x_h(\omega)\colon h=\sum_k 1_{C_k} h_k,\; 1\leq h_k\leq l_k \right\}, & \omega \in A_\ast, \\
		\text{cl}\left\{x_h(\omega)\colon h=h 1_\Omega \in L^0(\mathbb{N}), h \in \mathbb{N}\right\}, & \omega \in A^c_\ast. 
		\end{cases}
		\]
		By Theorem \ref{castaing}, the map $S$ is Effros measurable and closed-valued. 
		\item[(ii)] Inspection shows that $X_{S_X}\subset  X$. 
		Suppose there exists $x_0\in X$ such that $x_0\not\in X_{S_X}$. 
		Then there is $n\in L^0(\mathbb{N})$ such that $B_{1/m}(x_0)\cap X_{S_X}=\emptyset$ for all $m\geq n$. 
		This contradicts the construction of $S_X$. 
		Hence $X\subset  X_{S_X}$. 
		By the previous, $X_{S_{(X_S)}}=X_S$. 
		It follows from Theorem \ref{castaing} that $S=S_{X_S}$ as well. 
	\end{itemize}
\end{proof}
\begin{remark}
	There exist characterization results that are related to Theorem \ref{t:selec}, see e.g.~\cite[Theorem 2.1.6]{molchanov2005theory} and \cite[Theorem 2.3]{lepinette2016risk}  and the references therein for a background.  
    In this remark, we discuss how Theorem \ref{t:selec} relates to these results. 
	Let $E$ be a separable Banach space, and let $L^p(E)$ be the Bochner space of all $p$-integrable functions $x\colon \Omega\to E$ for $p\in [1,\infty]$. 
	For a set-valued mapping $S\colon \Omega\rightrightarrows E$, denote by $X_S^p:=X_S\cap L^p(E)$ the set of $p$-integrable selections of $S$.    
	Let $X\subset L^p(E)$ be norm-closed. 
	By \cite[Theorem 2.1.6]{molchanov2005theory}, $X=X_S^p$ for an Effros measurable closed-valued mapping $S\colon \Omega\rightrightarrows E$ if and only if $X$ is finitely decomposable\footnote{The property ``stability under (countable) gluings'' is known under the name ``decomposability'' in measurable selections and random set theory, see e.g.~\cite{molchanov2005theory,pennanen2011convex,rockafellar1976integral}, where it is usually employed in finite form, i.e.~stability w.r.t.~gluings along finite partitions.}. 
	An extension of  \cite[Theorem 2.1.6]{molchanov2005theory} to the case $p=0$, when $L^0(E)$ is endowed with the metric of convergence in probability, can be found in \cite[Theorem 2.3]{lepinette2016risk}.    
	In both cases ($p\in [1,\infty]$ and $p=0$) it can be verified that if a set $X\subset L^p(E)$ is decomposable and closed, then it is also infinitely decomposable, that is stable.  
	Therefore, the previous proposition extends the aforementioned results to the case that $E$ is a Polish space and convergence in norm or probability is replaced by almost sure convergence. 
\end{remark}
A frequently employed concept in stochastic optimal control is a normal integrand (see e.g.~\cite{pennanen2012stochastic,rockafellar1983deterministic} and the references therein), that is a function $f\colon \Omega\times E\to \overline{\mathbb{R}}$ whose epigraphical mapping $S_f\colon \Omega\rightrightarrows E\times \mathbb{R}$, $S_f(\omega):=\{(x,r)\in E\times \mathbb{R}\colon f(\omega,x)\leq r\}$, is closed-valued and Effros measurable. 
A consequence of normality of an integrand is that $f(\omega,x(\omega))$ is measurable in $\omega$ whenever $x\colon \Omega\to E$ is a measurable function. 
Moreover,  a normal integrand $f(\omega,x)$ is measurable in $\omega$ for fixed $x$ and lower semi-continuous in $x$ for fixed $\omega$ (cf.~\cite[Proposition 14.28]{rockafellar02}).  
We obtain the following ``functional'' version of Theorem \ref{t:selec}, where two normal integrands $f\colon \Omega\times E\to \overline{\mathbb{R}}$ and $g\colon \Omega\times E\to \overline{\mathbb{R}}$ are considered as identical if their epigraphical mappings coincide a.s.   
\begin{corollary}\label{c:normalintegrand}
	Let  $u\colon L^0(E)\to \bar{L}^0$ be stable and sequentially lower semi-continuous and let $f\colon \Omega\times E\to \overline{\mathbb{R}}$ be a normal integrand. 
	Then there exist a stable and sequentially lower semi-continuous  function $u_f\colon L^0(E)\to \bar{L}^0$ and a normal integrand $f_u\colon \Omega\times E\to \overline{\mathbb{R}}$ such that 
	$u_{f_u}=u$ and $f_{u_f}=f$. 
\end{corollary}
\begin{proof}
	Due to normality, $u_f\colon L^0(E)\to \bar{L}^0$ given by $x\mapsto (\omega\mapsto f(\omega,x(\omega)))$ is well defined. 
	Direct inspection shows that $u_f$ is stable and sequentially lower semi-continuous. 
	Conversely, put $X:=\{(x,r)\in L^0(E\times\mathbb{R})\colon u(x)\geq r\}$. 
	By assumption, $X$ is a stable and sequentially closed subset of $L^0(E\times \mathbb{R})$.  
	By Proposition \ref{t:selec}, there exist an Effros measurable and closed-valued map $S_X\colon \Omega\rightrightarrows E\times \mathbb{R}$ corresponding to $X$. 
	Thus $f_u\colon \Omega\times E\to \mathbb{R}\cup \{\pm\infty\}$ defined by $f(\omega,x):=\inf S(\omega)_x$ a.s.~is a normal integrand where $S(\omega)_x$ denotes the $x$-section of $S(\omega)$. 
	It follows from the reciprocality relations in Proposition \ref{t:selec} that $u_{f_u}=u$ and $f_{u_f}=f$. 
\end{proof}
We compare the assumptions which underly conditional analysis and measurable selections. 
Conditional analysis is applicable under the following two purely measure-theoretic hypotheses: 
\begin{itemize}
	\item A probability measure $\mathbb{P}$ on $(\Omega,\mathcal{F})$ needs to be fixed a priori in order to identify\footnote{Actually, one  needs to fix a $\sigma$-ideal $\mathcal{I}$ of $\mathcal{F}$ such that the quotient $\mathcal{F}/\mathcal{I}$ is a complete Boolean algebra. The ideal of null sets of a $\sigma$-finite measure is one such example, see \cite{drapeau2016algebra} and the references therein for more examples.} sets, functions, relations, etc.    
	\item One consequently works in the context of conditional sets \cite{drapeau2016algebra}. 
	In particular, all involved sets must satisfy stability w.r.t.~countable concatenations (cf.~Definition \ref{def:basic}).
\end{itemize} 
Conditional analysis does not rely on the following topological assumptions which are prevalent in measurable selections and random set theory: 
\begin{itemize}
	\item standard Borel space\footnote{For example, for purposes of a dynamical programming principle in finite discrete time stochastic optimal control, more precisely the existence of disintegration of measure, the underlying measure space is additionally assumed to be standard Borel in \cite{bertsekas1978stochastic}.}, measure completeness, closed-valued mappings and Polish spaces.     
\end{itemize}
The established connections in Theorem \ref{t:selec} and Corollary \ref{c:normalintegrand} suggest that a stochastic control problem can equally be formalized in the language of conditional set theory.  
In e.g.~\cite{pennanen2011convex,pennanen2012stochastic,pennanen2017shadow,rockafellar1974continuous,rockafellar1983deterministic} some form of integrability is always assumed which leads to further technicalities in the proofs, see also  \cite{hiai1977integrals,lepinette2017conditional,rockafellar1976integral} and the references therein for basic studies on the relations of (conditional) expectations and integrands.        
The main results in Section \ref{sec1} are established for general utilities which are not necessarily in the form of expected utilities, and no integrability assumptions are required.  
Moreover, the following features of our control sets distinguish us from the existing literature.  
\begin{itemize}
	\item We introduce a new notion of conditional compactness which works in finite and infinite dimensional settings thanks to a conditional version of the Heine-Borel theorem \cite[Theorem 4.6]{drapeau2016algebra}. Conditional compactness extends the notion of compact-valued and Effros measurable mappings, see \cite{jamneshan2017compact} where it is proved that conditional compactness uniquely corresponds to compact-valued and Effros measurable mappings in the finite dimensional case.  
	\item The control sets work in any conditional metric space. This involves many examples which are out of reach of the existing technology, for example conditional $L^p$-spaces on general probability spaces, $L^0(\mathbb{R})^n$ with a measurable dimension, and $L^0(X)$ where $X$ is a non-separable metric space. Another example are conditional weak topologies which are not included in this article for which conditional analysis offers extensive tools as well, see e.g.~\cite{drapeau2016algebra,jamneshan2017compact,zapata2016eberlein}.     
\end{itemize}

\end{document}